\documentclass[12pt,letterpaper,final,twoside,leqno]{amsart}
\usepackage{amsmath,amssymb,amsthm,amsfonts,amscd,amsopn}
\usepackage{hyperref}
\usepackage{eucal, mathrsfs}
\usepackage{amsmath,amssymb,amsthm,amsfonts,amscd,amsopn}
\usepackage{hyperref}
\usepackage{xspace}
\usepackage[all]{xy}\xyoption{dvips}
\usepackage{comment} 
\usepackage{paralist}
\usepackage{stmaryrd}
\usepackage{enumitem}
\usepackage{supertabular}
\DeclareMathAlphabet{\smallchanc}{OT1}{pzc}%
                                 {m}{it}
\DeclareFontFamily{OT1}{pzc}{}
\DeclareFontShape{OT1}{pzc}{m}{it}%
             {<-> s * [1.100] pzcmi7t}{}
\DeclareMathAlphabet{\mathchanc}{OT1}{pzc}%
                                 {m}{it}

\newcommand{\mcE}{\mathchanc{E}}

\newcommand{\mcH}{\mathchanc{H}}

\newcommand{\mcR}{\mathchanc{R}}

\newcommand{\mcm}{\mathchanc{m}}

\newcommand{\mco}{\mathchanc{o}}

\newcommand{\mct}{\mathchanc{t}}

\newcommand{\mcx}{\mathchanc{x}}

\DeclareFontFamily{OMS}{rsfs}{\skewchar\font'60}
\DeclareFontShape{OMS}{rsfs}{m}{n}{<-5>rsfs5 <5-7>rsfs7 <7->rsfs10 }{}
\DeclareSymbolFont{rsfs}{OMS}{rsfs}{m}{n}
\DeclareSymbolFontAlphabet{\scr}{rsfs}

\newcommand{\sF}{\scr{F}}
\newcommand{\sG}{\scr{G}}
\newcommand{\sH}{\scr{H}}
\newcommand{\sI}{\scr{I}}

\newcommand{\sL}{\scr{L}}

\newcommand{\sO}{\scr{O}}

\newcommand{\sfA}{{\sf A}}
\newcommand{\sfB}{{\sf B}}

\newcommand{\bC}{\mathbb{C}}

\newcommand{\bH}{\mathbb{H}}

\newcommand{\bN}{\mathbb{N}}

\newcommand{\bQ}{\mathbb{Q}}

\def\frm{\mathfrak{m}}

\newcommand{\ol}{\overline}

\newcommand{\into}{\hookrightarrow}

\newcommand{\wt}{\widetilde}

\newcommand{\rup}[1]{\left\lceil{#1}\right\rceil}
\newcommand{\rdown}[1]{\left\lfloor{#1}\right\rfloor}

\newcommand{\leteq}{\colon\!\!\!=}
\newcommand{\col}{\colon}

\newcommand{\ratmap}{\dasharrow}

\DeclareMathOperator{\codim}{codim}

\DeclareMathOperator{\depth}{{depth}}

\DeclareMathOperator{\exc}{Exc}
\DeclareMathOperator{\excnklt}{Exc_{nklt}}

\newcommand{\sExt}[0]{{\mcE\mcx\mct}}

\newcommand{\sHom}[0]{{\mcH\mco\mcm}}

\DeclareMathOperator{\red}{red}

\DeclareMathOperator{\Sing}{{Sing}}

\DeclareMathOperator{\Spec}{{Spec}}
\DeclareMathOperator{\supp}{{supp}}

\newcommand{\factor}[2]{\left. \raise 2pt\hbox{\ensuremath{#1}} \right/
        \hskip -2pt\raise -2pt\hbox{\ensuremath{#2}}}

\newcommand{\myR}{{\mcR\!}}

\newcommand{\blank}{\underline{\hskip 10pt}}

\newcommand{\kdot}{{{\,\begin{picture}(1,1)(-1,-2)\circle*{2}\end{picture}\ }}}
\newcommand{\mydot}{\kdot}

\newcommand{\cx}{\sf}

\def\coh#1.#2.#3.{H^{#1}(#2,#3)}
\def\dimcoh#1.#2.#3.{h^{#1}(#2,#3)}
\def\hypcoh#1.#2.#3.{\mathbb H_{\vphantom{l}}^{#1}(#2,#3)}
\def\loccoh#1.#2.#3.#4.{H^{#1}_{#2}(#3,#4)}
\def\dimloccoh#1.#2.#3.#4.{h^{#1}_{#2}(#3,#4)}
\def\lochypcoh#1.#2.#3.#4.{\mathbb H^{#1}_{#2}(#3,#4)}
\def\ses#1.#2.#3.{0  \longrightarrow  #1   \longrightarrow 
 #2 \longrightarrow #3 \longrightarrow 0} 
\def\sesshort#1.#2.#3.{0
 \rightarrow #1 \rightarrow #2 \rightarrow #3 \rightarrow 0}
\def\dist#1.#2.#3.{  #1   \longrightarrow 
 #2 \longrightarrow #3 \stackrel{+1}{\longrightarrow} } 
\def\CDdist#1.#2.#3.{  #1   @>>>  #2  @>>>   #3 @>+1>> }  
\def\shortses#1.#2.#3.{0  \rightarrow  #1   \rightarrow 
 #2  \rightarrow   #3 \rightarrow  0}
\def\shortdist#1.#2.#3.{  #1   \rightarrow 
 #2  \rightarrow   #3 \stackrel{+1}{\rightarrow} }  
\def\ddist#1.#2.#3.#4.#5.#6.{\CD
#1 @>>> #2 @>>> #3 @>+1>> \\
@VVV @VVV @VVV \\
#4 @>>> #5 @>>> #6 @>+1>> 
\endCD}
\def\ddistun#1.#2.#3.#4.#5.#6.{\CD
#1 @>>> #2 @>>> #3 @>+1>> \\
@. @VVV @VVV  \\
#4 @>>> #5 @>>> #6 @>+1>> 
\endCD}
\def\Iff#1#2#3{
\hfil\hbox{\hsize =#1
\vtop{\noin #2}
\hskip.5cm 
\lower.5\baselineskip\hbox{$\Leftrightarrow$}\hskip.5cm
\vtop{\noin #3}}\hfil\medskip}
\newcommand{\union}\cup
\newcommand{\intersect}\cap
\newcommand{\Union}\bigcup
\newcommand{\Intersect}\bigcap
\def\myoplus#1.#2.{\underset #1 \to {\overset #2 \to \oplus}}

\newcommand{\resto}{\big\vert_}

\def\qis{\,{\simeq}_{\text{qis}}\,}

\newcommand{\mypagesize}{
\textwidth= 6.25in
\textheight=8.75in
\voffset-.5in
\hoffset-.75in
\marginparwidth=56pt
}
\mypagesize

\begin{document}
\makeatletter
\newenvironment{refmr}{}{}
\renewcommand{\labelenumi}{{\rm (\thethm.\arabic{enumi})}}
%
\setitemize[1]{leftmargin=*,parsep=0em,itemsep=0.125em,topsep=0.125em}
\newcommand\james{M\raise .575ex \hbox{\text{c}}Kernan}

\renewcommand\thesubsection{\thesection.\Alph{subsection}}
\renewcommand\subsection{
  \renewcommand{\sfdefault}{pag}
  \@startsection{subsection}%
  {2}{0pt}{-\baselineskip}{.2\baselineskip}{\raggedright
    \sffamily\itshape\small
  }}
\renewcommand\section{
  \renewcommand{\sfdefault}{phv}
  \@startsection{section} %
  {1}{0pt}{\baselineskip}{.2\baselineskip}{\centering
    \sffamily
    \scshape
}}
\newcounter{lastyear}\setcounter{lastyear}{\the\year}
\addtocounter{lastyear}{-1}
\newcommand\sideremark[1]{%
\normalmarginpar
\marginpar
[
\hskip .45in
\begin{minipage}{.75in}
\tiny #1
\end{minipage}
]
{
\hskip -.075in
\begin{minipage}{.75in}
\tiny #1
\end{minipage}
}}
\newcommand\rsideremark[1]{
\reversemarginpar
\marginpar
[
\hskip .45in
\begin{minipage}{.75in}
\tiny #1
\end{minipage}
]
{
\hskip -.075in
\begin{minipage}{.75in}
\tiny #1
\end{minipage}
}}
\newcommand\Index[1]{{#1}\index{#1}}
\newcommand\inddef[1]{\emph{#1}\index{#1}}
\newcommand\noin{\noindent}
\newcommand\hugeskip{\bigskip\bigskip\bigskip}
\newcommand\smc{\sc}
\newcommand\dsize{\displaystyle}
\newcommand\sh{\subheading}
\newcommand\nl{\newline}
\newcommand\get{\input /home/kovacs/tex/latex/} 
\newcommand\Get{\Input /home/kovacs/tex/latex/} 
\newcommand\toappear{\rm (to appear)}
\newcommand\mycite[1]{[#1]}
\newcommand\myref[1]{(\ref{#1})}
\newcommand\myli{\hfill\newline\smallskip\noindent{$\bullet$}\quad}
\newcommand\vol[1]{{\bf #1}\ } 
\newcommand\yr[1]{\rm (#1)\ } 
\newcommand\cf{cf.\ \cite}
\newcommand\mycf{cf.\ \mycite}
\newcommand\te{there exist}
\newcommand\st{such that}
\newcommand\myskip{3pt}
\newtheoremstyle{bozont}{3pt}{3pt}%
     {\itshape}
     {}
     {\bfseries}
     {.}
     {.5em}
     {\thmname{#1}\thmnumber{ #2}\thmnote{ \rm #3}}
\newtheoremstyle{bozont-sf}{3pt}{3pt}%
     {\itshape}
     {}
     {\sffamily}
     {.}
     {.5em}
     {\thmname{#1}\thmnumber{ #2}\thmnote{ \rm #3}}
\newtheoremstyle{bozont-sc}{3pt}{3pt}%
     {\itshape}
     {}
     {\scshape}
     {.}
     {.5em}
     {\thmname{#1}\thmnumber{ #2}\thmnote{ \rm #3}}
\newtheoremstyle{bozont-remark}{3pt}{3pt}%
     {}
     {}
     {\scshape}
     {.}
     {.5em}
     {\thmname{#1}\thmnumber{ #2}\thmnote{ \rm #3}}
\newtheoremstyle{bozont-def}{3pt}{3pt}%
     {}
     {}
     {\bfseries}
     {.}
     {.5em}
     {\thmname{#1}\thmnumber{ #2}\thmnote{ \rm #3}}
\newtheoremstyle{bozont-reverse}{3pt}{3pt}%
     {\itshape}
     {}
     {\bfseries}
     {.}
     {.5em}
     {\thmnumber{#2.}\thmname{ #1}\thmnote{ \rm #3}}
\newtheoremstyle{bozont-reverse-sc}{3pt}{3pt}%
     {\itshape}
     {}
     {\scshape}
     {.}
     {.5em}
     {\thmnumber{#2.}\thmname{ #1}\thmnote{ \rm #3}}
\newtheoremstyle{bozont-reverse-sf}{3pt}{3pt}%
     {\itshape}
     {}
     {\sffamily}
     {.}
     {.5em}
     {\thmnumber{#2.}\thmname{ #1}\thmnote{ \rm #3}}
\newtheoremstyle{bozont-remark-reverse}{3pt}{3pt}%
     {}
     {}
     {\sc}
     {.}
     {.5em}
     {\thmnumber{#2.}\thmname{ #1}\thmnote{ \rm #3}}
\newtheoremstyle{bozont-def-reverse}{3pt}{3pt}%
     {}
     {}
     {\bfseries}
     {.}
     {.5em}
     {\thmnumber{#2.}\thmname{ #1}\thmnote{ \rm #3}}
\newtheoremstyle{bozont-def-newnum-reverse}{3pt}{3pt}%
     {}
     {}
     {\bfseries}
     {}
     {.5em}
     {\thmnumber{#2.}\thmname{ #1}\thmnote{ \rm #3}}
\theoremstyle{bozont}    
  \newtheorem{proclaim}{Theorem}[section]
\newtheorem{thm}[proclaim]{Theorem}
\newtheorem{mainthm}[proclaim]{Main Theorem}
\newtheorem{cor}[proclaim]{Corollary} 
\newtheorem{cors}[proclaim]{Corollaries} 
\newtheorem{lem}[proclaim]{Lemma} 
\newtheorem{prop}[proclaim]{Proposition} 
\newtheorem{conj}[proclaim]{Conjecture}
\newtheorem{subproclaim}[equation]{Theorem}
\newtheorem{subthm}[equation]{Theorem}
\newtheorem{subcor}[equation]{Corollary} 
\newtheorem{sublem}[equation]{Lemma} 
\newtheorem{subprop}[equation]{Proposition} 
\newtheorem{subconj}[equation]{Conjecture}
\theoremstyle{bozont-sc}
\newtheorem{proclaim-special}[proclaim]{\specialthmname}
\newenvironment{proclaimspecial}[1]
     {\def\specialthmname{#1}\begin{proclaim-special}}
     {\end{proclaim-special}}
\theoremstyle{bozont-remark}
\newtheorem{rem}[proclaim]{Remark}
\newtheorem{subrem}[equation]{Remark}
\newtheorem{notation}[proclaim]{Notation} 
\newtheorem{assume}[proclaim]{Assumptions} 
\newtheorem{obs}[proclaim]{Observation} 
\newtheorem{example}[proclaim]{Example} 
\newtheorem{examples}[proclaim]{Examples} 
\newtheorem{complem}[equation]{Complement}
\newtheorem{const}[proclaim]{Construction}   
\newtheorem{ex}[proclaim]{Exercise} 
\newtheorem{subnotation}[equation]{Notation} 
\newtheorem{subassume}[equation]{Assumptions} 
\newtheorem{subobs}[equation]{Observation} 
\newtheorem{subexample}[equation]{Example} 
\newtheorem{subex}[equation]{Exercise} 
\newtheorem{claim}[proclaim]{Claim} 
\newtheorem{inclaim}[equation]{Claim} 
\newtheorem{subclaim}[equation]{Claim} 
\newtheorem{case}{Case} 
\newtheorem{subcase}{Subcase}   
\newtheorem{step}{Step}
\newtheorem{approach}{Approach}
\newtheorem{fact}{Fact}
\newtheorem{subsay}{}
\newtheorem*{SubHeading*}{\SubHeadingName}%
\newtheorem{SubHeading}[proclaim]{\SubHeadingName}
\newtheorem{sSubHeading}[equation]{\sSubHeadingName}
\newenvironment{demo}[1] {\def\SubHeadingName{#1}\begin{SubHeading}}
  {\end{SubHeading}}%
\newenvironment{subdemo}[1]{\def\sSubHeadingName{#1}\begin{sSubHeading}}
  {\end{sSubHeading}} %
\newenvironment{demo-r}[1]{\def\SubHeadingName{#1}\begin{SubHeading-r}}
  {\end{SubHeading-r}}%
\newenvironment{subdemo-r}[1]{\def\sSubHeadingName{#1}\begin{sSubHeading-r}}
  {\end{sSubHeading-r}} %
\newenvironment{demo*}[1]{\def\SubHeadingName{#1}\begin{SubHeading*}}
  {\end{SubHeading*}}%
\newtheorem{defini}[proclaim]{Definition}
\newtheorem{question}[proclaim]{Question}
\newtheorem{subquestion}[equation]{Question}
\newtheorem{crit}[proclaim]{Criterion}
\newtheorem{pitfall}[proclaim]{Pitfall}
\newtheorem{addition}[proclaim]{Addition}
\newtheorem{principle}[proclaim]{Principle} 
\newtheorem{condition}[proclaim]{Condition}
\newtheorem{say}[proclaim]{}
\newtheorem{exmp}[proclaim]{Example}
\newtheorem{hint}[proclaim]{Hint}
\newtheorem{exrc}[proclaim]{Exercise}
\newtheorem{prob}[proclaim]{Problem}
\newtheorem{ques}[proclaim]{Question}    
\newtheorem{alg}[proclaim]{Algorithm}
\newtheorem{remk}[proclaim]{Remark}          
\newtheorem{note}[proclaim]{Note}            
\newtheorem{summ}[proclaim]{Summary}         
\newtheorem{notationk}[proclaim]{Notation}   
\newtheorem{warning}[proclaim]{Warning}  
\newtheorem{defn-thm}[proclaim]{Definition--Theorem}  
\newtheorem{convention}[proclaim]{Convention}  
\newtheorem*{ack}{Acknowledgment}
\newtheorem*{acks}{Acknowledgments}
\theoremstyle{bozont-def}    
\newtheorem{defn}[proclaim]{Definition}
\newtheorem{subdefn}[equation]{Definition}
\theoremstyle{bozont-reverse}    
\newtheorem{corr}[proclaim]{Corollary} 
\newtheorem{lemr}[proclaim]{Lemma} 
\newtheorem{propr}[proclaim]{Proposition} 
\newtheorem{conjr}[proclaim]{Conjecture}
\theoremstyle{bozont-reverse-sc}
\newtheorem{proclaimr-special}[proclaim]{\specialthmname}
\newenvironment{proclaimspecialr}[1]%
{\def\specialthmname{#1}\begin{proclaimr-special}}%
{\end{proclaimr-special}}
\theoremstyle{bozont-remark-reverse}
\newtheorem{remr}[proclaim]{Remark}
\newtheorem{subremr}[equation]{Remark}
\newtheorem{notationr}[proclaim]{Notation} 
\newtheorem{assumer}[proclaim]{Assumptions} 
\newtheorem{obsr}[proclaim]{Observation} 
\newtheorem{exampler}[proclaim]{Example} 
\newtheorem{exr}[proclaim]{Exercise} 
\newtheorem{claimr}[proclaim]{Claim} 
\newtheorem{inclaimr}[equation]{Claim} 
\newtheorem{SubHeading-r}[proclaim]{\SubHeadingName}
\newtheorem{sSubHeading-r}[equation]{\sSubHeadingName}
\newtheorem{SubHeadingr}[proclaim]{\SubHeadingName}
\newtheorem{sSubHeadingr}[equation]{\sSubHeadingName}
\newenvironment{demor}[1]{\def\SubHeadingName{#1}\begin{SubHeadingr}}{\end{SubHeadingr}}
\newtheorem{definir}[proclaim]{Definition}
\theoremstyle{bozont-def-newnum-reverse}    
\newtheorem{newnumr}[proclaim]{}
\theoremstyle{bozont-def-reverse}    
\newtheorem{defnr}[proclaim]{Definition}
\newtheorem{questionr}[proclaim]{Question}
\newtheorem{newnumspecial}[proclaim]{\specialnewnumname}
\newenvironment{newnum}[1]{\def\specialnewnumname{#1}\begin{newnumspecial}}{\end{newnumspecial}}
\numberwithin{equation}{proclaim}
\numberwithin{figure}{section} 
\newcommand\equinsect{\numberwithin{equation}{section}}
\newcommand\equinthm{\numberwithin{equation}{proclaim}}
\newcommand\figinthm{\numberwithin{figure}{proclaim}}
\newcommand\figinsect{\numberwithin{figure}{section}}
\newenvironment{sequation}{%
\numberwithin{equation}{section}%
\begin{equation}%
}{%
\end{equation}%
\numberwithin{equation}{thm}%
\addtocounter{thm}{1}%
}
\newcommand{\num}{\arabic{section}.\arabic{proclaim}}
\newenvironment{pf}{\smallskip \noindent {\sc Proof. }}{\qed\smallskip}
\newenvironment{enumerate-p}{
  \begin{enumerate}}
  {\setcounter{equation}{\value{enumi}}\end{enumerate}}
\newenvironment{enumerate-cont}{
  \begin{enumerate}
    {\setcounter{enumi}{\value{equation}}}}
  {\setcounter{equation}{\value{enumi}}
  \end{enumerate}}
\let\lenumi\labelenumi
\newcommand{\rmlabels}{\renewcommand{\labelenumi}{\rm \lenumi}}
\newcommand{\rmlabelsoff}{\renewcommand{\labelenumi}{\lenumi}}
\newenvironment{heading}{\begin{center} \sc}{\end{center}}
\newcommand\subheading[1]{\smallskip\noindent{{\bf #1.}\ }}
\newlength{\swidth}
\setlength{\swidth}{\textwidth}
\addtolength{\swidth}{-,5\parindent}
\newenvironment{narrow}{
  \medskip\noindent\hfill\begin{minipage}{\swidth}}
  {\end{minipage}\medskip}
\newcommand\nospace{\hskip-.45ex}
\makeatother

\begin{abstract}
  Irrational centers are defined analogously to associated primes. The union of
  irrational centers is the locus of non-rational singularities, but irrational
  centers carry more information. There may be embedded irrational centers signifying
  more complicated singularities. Various results regarding irrational centers are
  proved, in particular some concerning depth estimates and the Cohen-Macaulayness of
  certain ideal sheaves.  It is also proved that absolute irrational centers of a log
  canonical pair are also non-klt centers. This allows applying results proved for
  irrational centers for non-klt centers of log canonical pairs.
\end{abstract}

\title{Irrational centers}
\author{S\'andor J Kov\'acs}
\date{\today}
\thanks{Supported in part by NSF Grant DMS-0856185, and the Craig McKibben and Sarah
  Merner Endowed Professorship in Mathematics at the University of Washington.}
\address{University of Washington, Department of Mathematics, Box 354350,
Seattle, WA 98195-4350, USA} 
\email{skovacs@uw.edu}
\urladdr{http://www.math.washington.edu/$\sim$kovacs}
\maketitle
\newcommand{\szabores}{Szab\'o-resolution\xspace}
\newcommand{\DD}{D}

\centerline{\sf In memoriam Eckart Viehweg}

\bigskip\bigskip

\section{Introduction}

\noindent
Rational singularities form one of the most important classes of singularities.
Their essence lies in the fact that their cohomological behavior is very similar to
that of smooth points. For instance, vanishing theorems can be easily extended to
varieties with rational singularities.  Establishing that a certain class of
singularities is rational opens the door to using very powerful tools on varieties
with those singularities.

The main purpose of the present article is to get a handle on determining how far a
non-rational singularity is from being rational, or in other words, introduce a
measure of the failure of a singularity being rational.

Recently there has been an effort to extend the notion of rational singularities to
pairs. There are at least two approaches; Schwede and Takagi \cite{MR2492473} are
dealing with pairs $(X,\Delta)$ where $\rdown{\Delta}=0$ while Koll\'ar and Kov\'acs
\cite{KollarKovacsRP} are studying pairs $(X,\Delta)$ where $\Delta$ is reduced.  I
will work with pairs and concentrate on the latter approach, but the results are
interesting already in the classical case and should be easily adjustable to fit the
setup of the former approach.

I will introduce and start developing the notion of \emph{irrational centers} (or
\emph{non-rational centers}). These are special subvarieties of the singular locus
that are one way or another ``responsible'' for he failure of the singularity to be
rational. After having finished this article I was informed that Alexeev and Hacon
has introduced a similar notion in \cite{AH09}. The definition given here reduces to
their definition in the case $\Delta=0$. Some of their results are similar to the
ones in the present article, but their methods are different from those applied here.

The behaviour of irrational centers is very similar to that of \emph{non-klt
  centers}.  In fact, I will show that Koll\'ar's recent results
\cite{Kollar_local_KV} (cf.\ \cite{MR2435844}, \cite{Fujinobook}) concerning depth of
ideal sheaves of unions of non-klt centers has a reasonably close analogue for
irrational centers. In particular, I will prove the following results (for the
relevant definitions, see \S\ref{sec:pairs-gener-pairs} and
\S\ref{sec:irrational-centers}):

\begin{thm}[(= Corollary~\ref{cor:main-thm-1})]\label{thm:main-1}
  Let $(X,\DD)$ be a rational pair. Then $\sO_X(-D)$ is a CM sheaf.
\end{thm}

\begin{thm}[(= Corollary~\ref{thm:appl--theorem})]\label{thm:main-2}
  Let $(X,\DD)$ be a normal pair and $x\in X$ which is not the general point of an
  absolute irrational center of $(X,\DD)$. Then
  $$
  \depth_x\sO_X(-\DD)\geq \min ( 3, \codim_X x ). 
  $$
\end{thm}

The main focus of this article is the introduction of the notion of \emph{irrational
  centers} as a tool to study singularities. The above theorem is a demonstration of
how one may use this notion.
I also prove that 

\begin{thm}[(= Theorem~\ref{thm:irrational-is-non-klt})]
  Absolute irrational centers of lc pairs are also non-klt centers.
\end{thm}

This opens the door to numerous possible applications regarding log canonical pairs.
For instance it implies the following:

\begin{thm}[(= Corollary~\ref{cor:dlt}) 
  \protect{\cite[Thm.~111]{KollarKovacsRP},\cite[Thm.~2]{Kollar_local_KV}}] 
  Let $(X,\Delta)$ be a dlt pair. Then for any effective integral divisor $D\leq
  \rdown{\Delta}$ $\sO_X(-D)$ is a CM sheaf.
\end{thm}

\medskip

\begin{thm}[(= Corollary~\ref{cor:lc}) 
  \protect{\cite{MR2435844},\cite{Fujinobook},\cite{Kollar_local_KV}}] 
  Let $(X,\Delta)$ be an lc pair and $x\in X$ which is not a non-klt center of
  $(X,\Delta)$. Then
  $$
  \depth_x\sO_X(-\rdown{\Delta})\geq \min ( 3, \codim_X x ).  
  $$
\end{thm}

\medskip

\noindent
It should also be noted that for lc pairs Koll\'ar's results are more general than
the ones here. On one hand his results extend to more generally chosen integral
divisors and also to slc singularities. I believe both of those generalizations are
possible through the methods presented here, but this will be addressed at a later
time.

\begin{demo}{\bf Definitions and Notation}\label{demo:defs-and-not}
  Unless otherwise stated, all objects are assumed to be defined over $\bC$, all
  schemes are assumed to be of finite type over $\bC$ and a morphism means a morphism
  between schemes of finite type over $\bC$.

  For definitions related to \emph{pairs}, see \eqref{ssec:basic-definitions}.

  If $\phi:Y\to Z$ is a birational morphism, then $\exc(\phi)$ will denote the
  \emph{exceptional set} of $\phi$. By abuse of notation, if this exceptional set is
  of pure codimension $1$, then $\exc(\phi)$ will also denote the \emph{exceptional
    divisor} of $\phi$. For a closed subscheme $W\subseteq X$, the ideal sheaf of $W$
  is denoted by $\sI_{W\subseteq X}$ or if no confusion is likely, then simply by
  $\sI_W$.  For a point $x\in X$, $\kappa(x)$ denotes the residue field of
  $\sO_{X,x}$.
  

  For a proper birational morphism $\pi:Y\to X$ let $T\subseteq X$ denote the
  indeterminacy locus of the rational map $\pi^{-1}:X\ratmap Y$. Then for a subset
  $W\subseteq X$ we define the \emph{strict transform of $W$ on $Y$}, denoted by
  $\pi^{-1}_*W$ as the closure of $\pi^{-1}(W\setminus T)$ in $Y$. Notice that if $W$
  is contained in $T$, then its strict transform is the empty set.


  Let $X$ be a noetherian scheme, 
  $x\in X$ a (not necessarily closed) point, and $\sF$ a coherent sheaf on $X$.  The
  \emph{dimension} and \emph{codimension} of the closed subscheme
  $\overline{\{x\}}\subseteq X$ will be denoted by $\dim x$ and $\codim_X x$
  respectively.  In particular, $\dim X= \dim x + \codim_X x$ for any $x\in X$.  The
  \emph{dimension} of $\sF$ is the dimension of its support: $\dim\sF\leteq
  \dim\supp\sF$. The local dimension, denoted by $\dim_x$ is understood on the local
  scheme $(X,x)$ and it is equal to $\dim \sF_x$ the dimension of $\sF_x$ as an
  $\sO_{X,x}$-module.  The \emph{depth of $\sF$ at $x$}, denoted by $\depth_x\sF$ is
  defined as the depth of $\sF_x$ as an ${\sO_{X,x}}$-module. A non-zero coherent
  sheaf $\sF$ is said to satisfy \emph{Serre's condition $S_n$} if
  $$
  \depth_x\sF \geq \min(n, \dim_x\sF)
  $$
  for all $x\in X$ \cite[p.63]{MR1251956}. 
  Notice that this definition implies that if $X$ is contained in another noetherian
  scheme $Y$, then $\sF$ satisfies Serre's condition $S_n$ regarded as a sheaf on $X$
  if and only if it satisfies $S_n$ regarded as a sheaf on $Y$. This is because the
  depth as well as the support of $\sF$ is independent of the ambient scheme
  considered.

  The dualizing complex of $X$ is denoted by $\omega_X^\mydot$ and if $X$ is of pure
  dimension $n$ the dualizing sheaf of $X$ is defined as $\omega_X\leteq
  h^{-n}(\omega_X^\mydot)$.  Note that if $X$ is not normal, then this is not
  necessarily the push-forward of the canonical sheaf from the non-singular locus.

 Let $x\in X$ be a closed point.  Then $\sF$ is called \emph{Cohen-Macaulay}
 (\emph{CM} for short) \emph{at $x$} if $\sF_x$ is a Cohen-Macaulay module over
 $\sO_{X,x}$, and $\sF$ is \emph{Cohen-Macaulay} 
 (\emph{CM} for short) if it is CM at $x$ for all closed points $x\in \supp\sF$.  In
 particular, $X$ is CM if so is $\sO_X$.  Finally, $X$ is called \emph{Gorenstein} if
 $\sO_X$ is CM and $\omega_X$ is an invertible sheaf.

  A relatively straightforward consequence of the definition of the
  dualizing sheaf and basic properties of CM rings is that $X$ is CM
  if and only if $\omega_X^\mydot\qis \omega_X[n]$ cf.\
  \cite[3.5.1]{Conrad00}.

  Let $X$ be a complex scheme of dimension n. Let $D_{\rm filt}(X)$ denote the
  derived category of filtered complexes of $\sO_{X}$-modules with differentials of
  order $\leq 1$ and $D_{\rm filt, coh}(X)$ the subcategory of $D_{\rm filt}(X)$ of
  complexes $\cx K$, such that for all $i$, the cohomology sheaves of $Gr^{i}_{\rm
    filt}{\cx K}$ are coherent cf.\ \cite{DuBois81}, \cite{GNPP88}.  Let $D(X)$ and
  $D_{\rm coh}(X)$ denote the derived categories with the same definition except that
  the complexes are assumed to have the trivial filtration.  The superscripts $+, -,
  b$ carry the usual meaning (bounded below, bounded above, bounded).  Isomorphism in
  these categories is denoted by $\qis$.  A sheaf $\sF$ is also considered as a
  complex $\sF^\kdot$ with $\sF^0=\sF$ and $\sF^i=0$ for $i\neq 0$.  If ${\cx K}$
  is a complex in any of the above categories, then $h^i({\cx K})$ denotes the
  $i$-th cohomology sheaf of ${\cx K}$.

  The right derived functor of an additive functor $F$, if it exists, is denoted by
  $\myR F$ and $\myR^iF$ is short for $h^i\circ \myR F$. Furthermore, $\bH^i$,
  $\bH^i_{\rm c}$, $\bH^i_Z$ , and $\sH^i_Z$ will denote $\myR^i\Gamma$,
  $\myR^i\Gamma_{\rm c}$, $\myR^i\Gamma_Z$, and $\myR^i\sH_Z$ respectively, where
  $\Gamma$ is the functor of global sections, $\Gamma_{\rm c}$ is the functor of
  global sections with proper support, $\Gamma_Z$ is the functor of global sections
  with support in the closed subset $Z$, and $\sH_Z$ is the functor of the sheaf of
  local sections with support in the closed subset $Z$.  Note that according to this
  terminology, if $\phi\col Y\to X$ is a morphism and $\sF$ is a coherent sheaf on
  $Y$, then $\myR\phi_*\sF$ is the complex whose cohomology sheaves give rise to the
  usual higher direct images of $\sF$. $\sHom_X(\sF,\sG)$ denotes the sheaf of
  morphisms between the sheaves $\sF$ and $\sG$ and $\sExt^i_X\leteq
  \myR^i\sHom_X=h^i\circ \myR\sHom_X$.

  We will use the notion that a morphism ${f}: \sfA\to \sfB$ in a derived category
  \emph{has a left inverse}. This means that there exists a morphism $f^\ell: \sfB\to
  \sfA$ in the same derived category such that $f^\ell\circ{f}:\sfA\to\sfA$ is the
  identity morphism of $\sfA$. I.e., $f^\ell$ is a \emph{left inverse} of ${f}$.
\end{demo}

\section{Rational 
pairs}\label{sec:pairs-gener-pairs}

\subsection{Basic definitions}\label{ssec:basic-definitions}

A \emph{$\bQ$-divisor} is a $\bQ$-linear combination of integral Weil divisors;
$\Delta=\sum a_i\Delta_i$, $a_i\in\bQ$, $\Delta_i$ (integral) Weil divisor. For a
$\bQ$-divisor $\Delta=\sum a_i\Delta_i$ we will use the following notation:
$\Delta_{\red}\leteq\sum \Delta_i$ and the \emph{round-down} of $\Delta$ is defined
by the formula: $\rdown \Delta=\sum \rdown{a_i}\Delta_i$, where $\rdown{a_i}$ is the
largest integer not larger than $a_i$.

A \emph{log variety} or \emph{pair} $(X,\Delta)$ consists of an irreducible
variety (i.e., an irreducible reduced scheme of finite type over a field $k$) $X$ and
an effective $\bQ$-divisor $\Delta\subseteq X$ such that no irreducible component of
$\Delta$ is contained in $\Sing X$. (This last assumption is automatically satisfied
if, for example, $X$ is normal.)  A morphism of pairs $\phi:(Y,\Gamma)\to (X,\Delta)$
is a morphism $\phi:Y\to X$ such that $\phi(\supp \Gamma)\subseteq \supp\Delta$.

A \emph{reduced pair} is a pair $(X,\DD)$ where $\DD$ is a reduced integral divisor.
In this case with a slight abuse of notation we will use $\DD$ to also denote $\supp
D$.  If $(X,\DD)$ is a reduced pair, then $(X,\DD)$ is said to have \emph{simple
  normal crossings} or to be an \emph{snc pair at $p\in X$} if $X$ is smooth at $p$
and the components of $\DD$ are smooth at $p$ intersecting transversely in a Zariski
neighbourhood of $p$, i.e., in local analytic coordinates, $x_1,\dots,x_n$, near $p$,
$D$ is defined by $\prod_{i=1}^r x_i=0$ for some $r$. Furthermore, $(X,\DD)$ is
\emph{snc} if it is snc at every $p\in X$.  An snc pair $(X,\DD)$ will be called a
\emph{smooth pair} if the irreducible components of $D$ do not intersect.
Equivalently, $(X,\DD)$ is a smooth pair if and only if both $X$ and $D$ are smooth.

A morphism of pairs $\phi:(Y,\Delta_Y)\to (X,\Delta)$ is a \emph{partial log
  resolution of $(X,\Delta)$} if $\phi:Y\to X$ is a proper, birational morphism which
is an isomorphism near general points of $\Delta$ such that
$\Delta_Y=\phi^{-1}_*\Delta$.  A partial log resolution is a \emph{log resolution} if
$\exc(\phi)$ is a divisor and $\big(Y,(\Delta_Y)_{\red}+\exc(\phi)\big)$ is an snc
pair.  
Note that we allow $(X,\Delta)$ to be snc and still call a morphism with these
properties a log resolution. Also note that the notion of a log resolution is not
used consistently in the literature.

If $(X,\Delta)$ is a pair, then $\Delta$ is called a \emph{boundary} if
$\rdown{(1-\varepsilon)\Delta}=0$ for all $0<\varepsilon<1$, i.e., the coefficients
of all irreducible components of $\Delta$ are in the interval $[0,1]$.  For the
definition of \emph{klt, dlt}, and \emph{lc} pairs see \cite{KM98}.
  %
Let $(X, \Delta)$ be a pair and $\mu:X^{\rm m}\to X$ a proper birational morphism.
Let $E=\sum a_iE_i$ be the discrepancy divisor, i.e., a linear combination of
exceptional divisors such that
$$
K_{X^{\rm m}}+\mu^{-1}_*\Delta \sim_{\bQ} \mu^*(K_X + \Delta) + E
$$ and let
$\Delta^{\rm m}\leteq \mu^{-1}_*\Delta + \sum_{a_i\leq -1}E_i$.  For an irreducible
divisor $F$ on a birational model of $X$ we define its discrepancy as its coefficient
in $E$.  Notice that as divisors correspond to valuations, this discrepancy is
independent of the model chosen, it only depends on the divisor.  A \emph{non-klt
  place} of a pair $(X,\Delta)$ is an irreducible divisor $F$ over $X$ with
discrepancy at most $-1$ and a \emph{non-klt center} is the image of a non-klt place.
$\excnklt(\mu)$ denotes the union of the loci of all non-klt places of $\phi$.

Note that in the literature, non-klt places and centers are often called log
canonical places and centers. For a more detailed and precise definition see
\cite[p.37]{Hacon-Kovacs10}.
  
Now if $(X^{\rm m}, \Delta^{\rm m})$ is as above, then it is a \emph{minimal dlt
  model} of $(X,\Delta)$ if it is a dlt pair and the discrepancy of every
$\mu$-exceptional divisor is at most $-1$ cf.\ \cite{KK10}. Note that if $(X,\Delta)$
is lc with a minimal dlt model $(X^{\rm m}, \Delta^{\rm m})$, then $K_{X^{\rm
    m}}+\Delta^{\rm m} \sim_{\bQ} \mu^*(K_X+\Delta)$. Also note that minimal dlt
models are not unique; for instance, blowing up the intersection of two or more
irreducible components of an snc divisor produces a (new) minimal dlt model of the
given snc pair.

\subsection{Rational pairs} 

Recall the definitions of \emph{rational singularities}:

\begin{defn}\label{def:rtl-sing}
  Let $X$ be a normal variety and $\phi :Y \rightarrow X$ a resolution of
  singularities. $X$ is said to have \emph{rational} singularities if
  $\myR^i\phi_*\sO_Y=0$ for all $i>0$, or equivalently if the natural map $\sO_X\to
  \myR\phi_*\sO_Y$ is a quasi-isomorphism.
\end{defn}

\begin{defn}\label{def:normal-pair}
  Let $(X,\DD)$ be a pair and $\DD$ an integral divisor. Then $(X,\DD)$ is
  called a \emph{normal pair} if there exists a log resolution $\phi:(Y,\DD_Y)\to
  (X,\DD)$ such that the natural morphism $\phi^\#:\sO_X(- \DD)\to
  \phi_*\sO_Y(-{\DD_Y})$ is an isomorphism.
\end{defn}

\begin{prop}
  Let $(X,\DD)$ be a normal pair. Then $(X,\DD_{\red})$ is an snc pair in codimension
  $1$.
\end{prop}

\begin{proof}
  Let $\phi:(Y,\DD_Y)\to (X,\DD)$ be a log resolution for which the natural
  morphism $$\phi^\#:\sO_X(- \DD)\to \phi_*\sO_Y(-{\DD_Y})$$ is an isomorphism.

  First consider $X\setminus\supp\DD$. Then $\phi^\#:\sO_{X\setminus\supp\DD}\to
  \phi_*\sO_{Y\setminus\supp\phi^{-1}(\DD)}$ is an isomorphism, and hence
  $X\setminus\supp\DD$ is normal.  In particular, $X$ is smooth in codimension $1$
  away from $\DD$.

  Next observe that by definition $\phi$ is an isomorphism near the general points of
  $\DD$ and hence $(X,\DD_{\red})$ is an snc pair in codimension $1$ near $\DD$. This
  proves the claim.
\end{proof}

\begin{prop}\label{prop:normal-reduced-is-a-normal-pair}
  Let $(X,\DD)$ be a reduced pair. Then we have the following two implications:
  \begin{enumerate}
  \item if $X$ is normal and $\myR^1\phi_*\sO_Y(-\DD_Y)=0$ 
    for a log resolution $\phi:(Y,\DD_Y)\to (X,\DD)$ , then $(X,\DD)$ is a normal
    pair, and
    \label{item:normal1}
  \item if $(X,\DD)$ is a normal pair and $\DD$ is Cartier, then $X$ is normal.
    \label{item:normal2}
  \end{enumerate}
\end{prop}

\begin{proof}
  Consider the following diagram of exact sequences:
  $$
  \xymatrix{%
    0 \ar[r] & \sO_X(- \DD)\ar[r]\ar[d]_\alpha & \sO_X \ar[r]\ar[d]_\beta & \sO_{\DD}
    \ar[r]\ar[d]_\gamma & 0\\
    0 \ar[r] & \phi_*\sO_Y(-{\DD_Y}) \ar[r] & \phi_*\sO_Y \ar[r] & \phi_*\sO_{\DD_Y}
    \ar[r] & \myR^1\phi_*\sO_Y(-\DD_Y) }
  $$
  By definition both $X$ and $D$ are reduced, so $\beta$ and $\gamma$ are injective
  and then $\alpha$ is injective as well.  

  Recall that $X$ is normal if and only is $\beta$ is an isomorphism and $(X,\DD)$ is
  normal if and only is $\alpha$ is an isomorphism.

  If $\beta$ is an isomorphism, then it is in particular surjective and hence
  $\gamma$ is surjective. If furthermore $\myR^1\phi_*\sO_Y(-\DD_Y)=0$, then $\gamma$
  is an isomorphism and then by the Snake Lemma $\alpha$ is also an isomorphism. This
  proves (\ref{prop:normal-reduced-is-a-normal-pair}.\ref{item:normal1}).

  If $\alpha$ is an isomorphism and $\DD$ is Cartier then $\phi^*D\sim \DD_Y+F$ for
  some effective $\phi$-exceptional divisor $F\subset Y$.  Therefore $\alpha$ factors
  as
  $$
  \xymatrix{%
    \sO_X(-\DD)\ar[r]\ar@/_2pc/[rr]^\simeq_\alpha & \phi_*\phi^*\sO_X(-\DD) \ar[r] &
    \phi_*\sO_Y(-\DD_Y). 
  }
  $$
  Since these are torsion-free sheaves of rank $1$ it follows that all three must be
  isomorphic. In other words,
  $$
  \sO_X(-\DD)\simeq \phi_*\phi^*\sO_X(-\DD)\simeq \sO_X(-\DD)\otimes \phi_*\sO_Y
  $$
  and hence $\beta$ is an isomorphism, so
  (\ref{prop:normal-reduced-is-a-normal-pair}.\ref{item:normal2}) is proven.
\end{proof}

\begin{defn}\cite{KollarKovacsRP}
  \label{def:rat-sing-pairs}
  A reduced pair $(X,\DD)$ is called a \emph{rational pair} if there exists a log
  resolution $\phi:(Y,\DD_Y)\to (X,\DD)$ such that
  \begin{equation*}
    \sO_X(-\DD)\qis\myR\phi_*\sO_Y(-{\DD_Y}).
  \end{equation*}
  A log resolution as above will be called a \emph{rational log resolution of
    $(X,\DD)$}.
\end{defn}

\begin{rem}
  From the definition it is not obvious whether $(X,\DD)$ being rational implies that
  $X$ has rational singularities.  It turns out that this is actually true if either
  $X$ is Cohen-Macaulay or $D$ is Cartier \eqref{thm:rtl-pairs-are-rtl}.
\end{rem}

\begin{rem}
  Notice that this definition is different from that of Schwede and Takagi's
  \cite{MR2492473} in several ways. In particular, according to this definition snc
  pairs are rational. On the other hand this definition is not independent of the
  resolution chosen but the only way it depends on the resolution is exactly the fact
  that it allows snc pairs to be rational. This situation is similar to the case of
  dlt singularities, where a dlt pair may have log canonical centers even though it
  has no ``purely'' log canonical singularities. In fact, the irrational centers
  (cf.\ \eqref{def:irrational-centers}) are exactly the log canonical (or non-klt)
  centers of a dlt pair.

  In Schwede and Takagi's terminology a rational pair with $D$ Cartier is said to
  have \emph{purely rational singularities}. See \cite[3.15]{MR2492473} for
  details. 
\end{rem}

\begin{example}\label{ex:dlt-is-rtl}
  Let $(X,\Delta)$ be a $\mathbb Q$-factorial dlt pair. Then it follows from
  \cite[111]{KollarKovacsRP} or
  \cite[2.1]{Kovacs10a} that $(X,\rdown{\Delta})$ is a rational pair.
\end{example}

The following is a simple consequence of known vanishing theorems that appear in
various forms in \cite{MR1993751,Fujinobook,KollarKovacsRP}. Since only this simple
version is needed here, a reasonably self-contained proof is provided for the
convenience of the reader. This statement was independently observed by Zsolt
Patakfalvi.

\begin{thm}
  [(Grauert-Riemenschneider vanishing for pairs) %
  {\rm cf.\ \cite{MR1993751,Fujinobook,KollarKovacsRP}}] %
  \label{thm:rel-kv-vanish}
  Let $(X,\DD)$ be a pair and $\phi:(Y,\DD_Y)\to (X,\DD)$ a log resolution. Let
  $B\leq \DD_Y$ be an effective reduced integral divisor.  Then
  $$\myR^i\phi_*\omega_Y(B)=0\qquad\text{ for $i>0$.}$$
\end{thm}

\begin{subrem}
  Notice that the statement \emph{does not} say that vanishing holds for \emph{any}
  pairs. It is important that $\DD_Y$ is the strict transform of $\DD$.
\end{subrem}

We need the following simple generalization of the (relative) Kawamata-Viehweg
vanishing theorem.  It is contained implicitly in \cite{MR1993751} and explicitly in
\cite[2.33]{Fujinobook} and \cite[4.3]{KollarKovacsRP}.

\newcommand{\BB}{B}

\begin{lem}\label{lem:GR}
  Let $Y$ be a smooth variety, $\phi:Y\to X$ a proper morphism, and $M$ a $\phi$-nef
  and $\phi$-big $\bQ$-divisor on $Y$. Further let $L$ be a Cartier divisor with
  $\sL=\sO_Y(L)$ and $\BB+\sum \Delta_i$ an effective simple normal crossing divisor
  on $Y$.  Assume that for some $0\leq a_i< 1$, $L\equiv M+\sum a_i\Delta_i$ and that
  $M$ is $\phi$-big on all log canonical centers of $(Y,\BB)$.  Then
  $$
  \myR^if_*\big(\omega_Y(\BB)\otimes \sL\big)=0\qquad \text{for $i>0$}.
  $$
\end{lem}

\begin{proof}
  Write $\BB=\BB_1+\BB'$ where $\BB_1$ is irreducible and consider the
  following short exact sequence,
  $$
  0\to \omega_Y(\BB')\otimes \sL\to \omega_Y(\BB)\otimes \sL \to \omega_{\BB_1}(
  \BB'\resto{\BB_1})\otimes \sL\resto{\BB_1} \to 0.
  $$

  Notice that the log canonical centers of $(\BB_1, \BB'\resto{\BB_1})$ are
  restrictions of the log canonical centers of $(Y,\BB)$. Therefore, using the long
  exact cohomology sequence of $\myR \phi_*$ and induction on the number of
  components of $\BB$ and the dimension of $Y$ this reduces the statement to the
  usual Kawamata-Viehweg vanishing theorem.
\end{proof}

\begin{proof}[Proof of \eqref{thm:rel-kv-vanish}]
  Since $\phi:(Y,\DD_Y)\to (X,\DD)$ is a log resolution, it follows that $\DD_Y+E$ is
  supported on an snc divisor where $E=\exc(\phi)$ is the exceptional divisor of
  $\phi$. Therefore any log canonical center of $(Y,\BB)$, in other words any
  intersection of the irreducible components of $\BB$, intersects $E$ transversally.
  In particular $\sO_Y$ is $\phi$-big on any log canonical center of $(Y,\BB)$ and
  hence the statement follows from \eqref{lem:GR}.
\end{proof}

\begin{cor}\label{cor:rtl-pairs-0}
  Let $(X,D)$ be a rational pair and $\phi:(Y,\DD_Y)\to (X,\DD)$ a rational
  resolution. 
  \begin{enumerate}
  \item $\sO_X(-\DD)\simeq \phi_*\sO_Y(-{\DD_Y})$, i.e., $(X,\DD)$ is
    normal,
  \item $\myR^i\phi_*\sO_Y(-{\DD_Y})=0$ for $i>0$, 
  \item $\myR^i\phi_*\omega_Y({\DD_Y})=0$ for $i>0$. 
    \label{item:rtl-pairs-iv}
  \end{enumerate}
\end{cor}

\begin{thm}
  \label{thm:rtl-pairs-are-rtl}
  Let $(X,\DD)$ be a rational pair and $\phi:(Y,\DD_Y)\to (X,\DD)$ a rational
  resolution.  Then 
  \begin{enumerate}
  \item $\phi_*\omega_Y(\DD_Y)\simeq \omega_X(D)\leteq\sHom_X(\sO_X(-\DD),\omega_X)$,
    and \label{item:rtl-is-rtl1}
  \item $\phi_*\omega_Y\simeq \omega_X$. \label{item:rtl-is-rtl2}
  \end{enumerate}
\end{thm}

\begin{proof} 
  Let $n=\dim X$. Then Grothendieck duality and
  (\ref{cor:rtl-pairs-0}.\ref{item:rtl-pairs-iv}) yields that
  \begin{multline}\tag{$\star$}
    \myR\sHom_X(\sO_X(-\DD),\omega_X^\mydot)\qis
    \myR\sHom_X(\myR\phi_*\sO_Y(-\DD_Y),\omega_X^\mydot)\qis \\ \qis \myR
    \phi_*\myR\sHom_Y(\sO_Y(-\DD_Y), \omega_Y^\mydot)\qis \myR
    \phi_*\omega_Y(\DD_Y)[n] \qis \phi_*\omega_Y(\DD_Y)[n]
  \end{multline}
  and 
  \begin{equation*}
    \tag{$\star\star$}
    \omega_D^\mydot \qis  \myR\sHom_X(\sO_{\DD},\omega_X^\mydot).
  \end{equation*}
  Observe that $(\star)$ immediately implies
  (\ref{thm:rtl-pairs-are-rtl}.\ref{item:rtl-is-rtl1}). 
  Consider the short exact sequence
  $$0\to \sO_X(-\DD)\to \sO_X \to \sO_{\DD}\to 0$$
  and apply the functor $\myR\sHom_X(\__,\omega_X^\mydot)$ to obtain the
  distinguished triangle 
  \begin{equation*}
    \xymatrix{%
      \myR\sHom_X(\sO_{\DD},\omega_X^\mydot)
      \ar[r] & \omega_X^\mydot \ar[r] &
      \myR\sHom_X(\sO_X(-\DD),\omega_X^\mydot) \ar[r]^-{+1}&. 
    }
  \end{equation*}
  By $(\star)$ and $(\star\star)$ this is the same as 
  \begin{equation*}
    \xymatrix{%
      \omega_{\DD}^\mydot
      \ar[r] & \omega_X^\mydot \ar[r] &
      \phi_*\omega_Y(\DD_Y)[n]
      \ar[r]^-{+1}&. 
    }
  \end{equation*}
  Next consider the long exact sequence of cohomology sheaves induced by this
  distinguished triangle:
  $$
  0\to \omega_X \to \phi_*\omega_Y(\DD_Y) \to \omega_{\DD} \to\dots
  $$
  Since $Y$ is smooth 
  we have a similar short exact sequence on $Y$.
  $$
  0\to \omega_Y \to \omega_Y(\DD_Y) \to \omega_{\DD_Y} \to 0
  $$
  Applying $\phi_*$ and Grauert-Riemenschneider vanishing one obtains a commutative
  diagram of exact sequences:
  $$
  \xymatrix{%
    0 \ar[r] & \phi_*\omega_Y\ar[d]_\alpha \ar[r] & \phi_*\omega_Y({\DD_Y})
    \ar[r]\ar[d]_\beta^\simeq & \phi_*\omega_{\DD_Y} \ar[r] \ar[d]_\gamma & 0\\
    0 \ar[r] & \omega_X\ar[r] & \phi_*\omega_Y({\DD_Y}) \ar[r]^-\tau & \omega_{\DD}  }
  $$
  Since $\phi$ is an isomorphism at the general points of $Y$ and the irreducible
  components of $\DD_Y$ and since $\omega_Y$ and $\omega_{\DD_Y}$ are torsion-free,
  it follows that both $\alpha$ and $\gamma$ are injective. Then the fact that
  $\beta$ is an isomorphism implies that the image of $\tau$ is
  $\phi_*\omega_{\DD_Y}\subseteq \omega_D$. However, that implies that then $\ker
  \tau \simeq \phi_*\omega_Y$ and hence $\alpha$ is an isomorphism.  
\end{proof}

\begin{cor}[\protect{\cite[3.20]{MR2492473}}]
  Let $(X,\DD)$ be a rational pair.  Assume that $\DD$ is a Cartier divisor.  Then
  $X$ has only rational singularities and in particular it is Cohen-Macaulay.
\end{cor}

\begin{proof}
  If $\DD$ is Cartier, then $h^i(\omega_X^\mydot)=0$ for $i\neq\dim X$ by $(\star)$
  and hence $X$ is Cohen-Macaulay.  It also follows that $X$ is normal by
  (\ref{prop:normal-reduced-is-a-normal-pair}.\ref{item:normal2}) and hence satisfies
  Kempf's criterion for rational singularities by
  (\ref{thm:rtl-pairs-are-rtl}.\ref{item:rtl-is-rtl2}).

  This proves the statement, but actually one can give a simple direct proof:

  Let $\phi:(Y,\DD_Y)\to (X,\DD)$ be a rational resolution and $n=\dim X$.  By
  assumption the natural map $\sO_X(-\DD)\to \myR\phi_*\sO_Y(-\DD_Y)$ is a
  quasi-isomorphism and since $\phi^*\DD-\DD_Y$ is an effective (exceptional) divisor
  it factors the following way:
  $$
  \xymatrix{%
    \sO_X(-\DD)\ar[r]\ar@/_1.5pc/[rr]_\qis & \myR\phi_*\phi^*\sO_X(-\DD) \ar[r] &
    \myR\phi_*\sO_Y(-\DD_Y). 
  }
  $$
  Twisting by $\sO_X(\DD)$ implies that the natural morphism $\sO_X\to
  \myR\phi_*\sO_Y$ has a left inverse. Then $X$ has rational singularities by
  \cite[Thm.~1]{Kovacs00b}.
\end{proof}

\section{Depth and duality}

The main statement in this section is a simple reformulation of Grothendieck's
vanishing and non-vanishing theorems of certain local cohomology groups
characterizing depth in terms of similar vanishing and non-vanishing involving the
dualizing complex. 

Throught the article dualizing complexes will be considered \emph{normalized} so if
$X$ is generically non-reduced then $\supp h^i(\omega_X^\mydot)\subseteq \Sing X$ for
$i\neq -\dim X$.

  First we need an auxiliary result on the localization of dualizing complexes. This
  is undoubtedly known to experts. A proof is included for the benefit of the
  uninitiated reader.

\begin{lem}
  \label{lem:localizing-omega}
  Let $X$ be a scheme that admits a dualizing complex $\omega_X^\mydot$ (this holds
  for instance if $X$ is of finite type over a field).  Let $X_x\simeq \Spec
  \sO_{X,x}$ denote the local scheme of $X$ at $x$. Then $X_x$ admits a dualizing
  complex and $$\omega_{X_x}^\mydot\simeq \omega_X^\mydot\otimes \sO_{X,x}[-\dim x].$$
\end{lem}

\begin{proof}
  As the statement is local we may assume that $X$ is embedded into a Gorenstein
  scheme as a closed subscheme by \cite[1.4]{MR1859029}.  Let $j:X\into Y$ be such an
  embedding, $N=\dim Y$ and $m= \codim_Y x= \dim \sO_{Y,x}$.  Then by Grothendieck
  duality and because $Y$ is CM and $j$ is a closed embedding,
  \begin{multline*}
    \omega_X^\mydot\qis \myR\sHom _X(\sO_X,\omega_X^\mydot)\qis \myR j_*\myR\sHom
    _X(\sO_X,\omega_X^\mydot)\qis \\
    \underbrace{\qis}_{\hskip-40pt\text{Grothendieck duality}\hskip-40pt} \myR\sHom
    _Y(\myR j_*\sO_X,\underbrace{\omega_Y^\mydot}_{\hskip-40pt\text{dualizing
        complex}\hskip-40pt})\qis \myR\sHom
    _Y(\sO_X,\underbrace{\omega_Y}_{\hskip-40pt\text{dualizing sheaf}\hskip-40pt})[N].
  \end{multline*}
  Then taking cohomology and localizing at $x$ gives that 
  \begin{multline*}
    \omega_X^\mydot\otimes \sO_{X,x}[-\dim x]
    \qis \big(\omega_X^\mydot\big)_x[-\dim x]\qis
    \\ \qis
    \big(\myR\sHom _Y(\sO_X,\omega_Y)[N-\dim x]\big)_x\qis 
    \myR\sHom _{Y_x}(\sO_{X,x},\omega_{Y,x}[m]) \qis 
    \\ \qis
    \myR\sHom _{Y_x}(\myR(j_x)_*\sO_{X,x},\omega_{Y,x}^\mydot)\qis 
    \myR(j_x)_*\myR\sHom _{X_x}(\sO_{X,x},\omega_{X,x}^\mydot)\qis
    \omega_{X,x}^\mydot
  \end{multline*}

\end{proof}

\begin{prop}\label{thm:depth-ext-criterion}
  Let $X$ be a scheme that admits a dualizing complex $\omega_X^\mydot$.
  Let $x\in X$ be a point and $\sF$ a coherent sheaf on $X$. Let $d=\dim_x \sF + \dim
  x$ and $t=\depth_x\sF+\dim x$.
  Then 
  $$
  \big(\sExt^{-i}_X(\sF,\omega_X^\mydot)\big)_x=0\qquad \text{ for $i>d$ and
    $i<t$.} 
  $$
  Furthermore, if $\sF_x\neq 0$, then
  $$
  \big(\sExt^{-d}_X(\sF,\omega_X^\mydot)\big)_x\neq 0, 
  \quad\text{and}\quad
  \big(\sExt^{-t}_X(\sF,\omega_X^\mydot)\big)_x\neq 0, 
  $$
\end{prop}

\begin{proof}
  We may obviously assume that $\sF\neq 0$.  Localization is exact and commutes with
  the $\sHom$ functor, so $\big(\sExt^{-i}_X(\sF,\omega_X^\mydot)\big)_x\simeq
  \sExt^{-i}_{X_x}\left(\sF_x,\big(\omega_X^\mydot\big)_x\right)$ and then the latter
  group is isomorphic to $\sExt^{\dim
    x-i}_{X_x}\left(\sF_x,\omega_{X_x}^\mydot\right)$.  This is the Matlis dual of
  $H^{i-\dim x}_x(Y,\sF)$ by \cite[V.6.2]{MR0222093}.  Therefore we obtain that
  $$
  \big(\sExt^{-i}_X(\sF,\omega_X^\mydot)\big)_x=0 \quad\Leftrightarrow\quad
  H^{i-\dim x}_x(Y,\sF)=0
  $$
  and since both $\depth_x\sF$ and $\dim_x\sF$ remain the same over $Y$, the
  statement follows from Grothendieck's theorem \cite[3.5.7]{MR1251956}.
\end{proof}

\begin{cor}
  \label{cor:CM-criterion}
  Under the same conditions and using the same notation as in
  \eqref{thm:depth-ext-criterion} one has that $\sF$ is CM at $x\in X$ if and only if
  $$
  \big(\sExt^{-i}_X(\sF,\omega_X^\mydot)\big)_x=0\qquad \text{ for $i\neq d= \dim_x
    \sF + \dim x$.}
  $$
\end{cor}

\bigskip

\begin{cor}
  \label{cor:main-thm-1}
  Let $(X,\DD)$ be a rational pair. Then $\sO_X(-D)$ is a CM sheaf.
\end{cor}

\begin{proof}
  Consider a rational log resolution $\phi:(\wt X,\wt D)\to (X,\DD)$ (cf.\
  \eqref{def:rat-sing-pairs}).  Using the assumption and Grothendieck duality we
  obtain the following:
  \begin{multline*}
    \myR\sHom_X(\sO_X(-D),\omega_X^\mydot)\qis %
    \myR\sHom_{X}(\myR\phi_*\sO_{\wt X}(-\wt D),\omega_X^\mydot) \qis \\ \qis %
    \myR\phi_*\myR\sHom_{\wt X}(\sO_{\wt X}(-\wt D),\omega_{\wt X}^\mydot) \qis %
    \myR\phi_*\omega_{\wt X}^\mydot(\wt D) \qis \phi_*\omega_{\wt X}(\wt D)[\dim
    X].
  \end{multline*}
  This implies that
  $$
  \sExt^{-i}_X(\sO_X(-D),\omega_X^\mydot)=0\qquad \text{ for $i\neq \dim X$}
  $$
  and hence the statement follows by \eqref{cor:CM-criterion}.
\end{proof}

\begin{cor}\label{cor:dlt}
  Let $(X,\Delta)$ be a dlt pair. Then for any effective integral divisor $D\leq
  \rdown{\Delta}$ $\sO_X(-D)$ is a CM sheaf.
\end{cor}

\begin{proof}
  By \cite[Thm.~111]{KollarKovacsRP} $(X,D)$ is a rational pair and hence the
  statement follows from \eqref{cor:main-thm-1}.
\end{proof}

\section{Irrational centers}\label{sec:irrational-centers}

\begin{defn}
  Let $\sF$ be a coherent sheaf on a scheme $X$. Then 
  $x\in X$ is an \emph{associated point of $\sF$} if the maximal ideal
  $\frm_{X,x}\subset \sO_{X,x}$ is an associated prime of the module $\sF_x$. In
  other words, $x\in X$ is an associated point of $\sF$ if the maximal ideal
  $\frm_{X,x}\subset \sO_{X,x}$ consists of zero-divisors of the module $\sF_x$.
\end{defn}

Now we are ready to make the definition of the namesake of the present article: 

\begin{defn}\label{def:irrational-centers}
  Let $(X,\DD)$ be a reduced pair and 
  $\phi: (\wt X,\wt \DD)\to (X,\DD)$ a log resolution.  If $x\in X$ is an associated
  point of $\myR^i\phi_*\sO_{\wt X}(-\wt\DD)$ 
  for some $i>0$, then we call $Z=\ol{\{x\}}$, the Zariski closure of $\{x\}$, a
  \emph{relative irrational 
    center of $(X,\DD)$ with respect to $\phi$}.  
  A closed subset $Z\subseteq X$ is  called an \emph{irrational center of
  $(X,\DD)$} if there exists a log resolution  $\phi: (\wt X,\wt \DD)\to (X,\DD)$
  such that $Z$ is a relative irrational center of  $(X,\DD)$ with respect to
  $\phi$ 
  and $Z\subseteq X$ is called an \emph{absolute irrational center} if
  for any log resolution $\phi: (\wt X,\wt \DD)\to (X,\DD)$, $Z$ is a relative
  irrational center of $(X,\DD)$ with respect to $\phi$.
\end{defn}

\begin{rem}
  If $\DD=0$ then these definitions coincide and agree with the one given in
  \cite{AH09}. This can be proven the same way as one proves that having rational
  singularities does not depend on the resolution chosen.
\end{rem}

\section{Depth estimates}

\subsection{Truncated functors and distinguished triangles}

Let $\phi:(\wt X,\wt \DD)\to (X,\DD)$ be a log resolution.

Now define a series of derived category objects recursively as follows.  Let
$\myR^{\geq 0}\phi_*\leteq \myR\phi_*$ and consider the natural transformation
$$
\xymatrix{%
  \phi_* \ar[r] & \myR\phi_*.  }
$$
Then let $\myR^{\geq 1}\phi_*\sO_{\wt X}(-\wt\DD)$ be defined as the object
completing the induced natural morphism to a distinguished triangle\footnote{This is
  essentially the truncated complex of $\myR\phi_*\sO_{\wt X}(-\wt\DD)$.  The shift
  is included to make the subsequent definitions simpler and more balanced.}:
$$
\xymatrix{%
  \phi_*\sO_{\wt X}(-\wt\DD) \ar[r] & \myR\phi_*\sO_{\wt X}(-\wt\DD) \ar[r] &
  \myR^{\geq 1}\phi_*\sO_{\wt X}(-\wt\DD)[-1] \ar[r]^-{+1} & .  }
$$
By construction there exists a natural morphism,
$$
\xymatrix{%
  \myR^1\phi_*\sO_{\wt X}(-\wt\DD) \ar[r] & \myR^{\geq 1}\phi_*\sO_{\wt X}(-\wt\DD), }
$$
and we let $\myR^{\geq 2}\phi_*\sO_{\wt X}(-\wt\DD)$ be defined as the object
completing the above natural morphism to a distinguished triangle and so on to obtain
a series of objects and distinguished triangles for each $p\in\bN$:
$$
\xymatrix{%
  \myR^p\phi_*\sO_{\wt X}(-\wt\DD) \ar[r] & \myR^{\geq p}\phi_*\sO_{\wt X}(-\wt\DD)
  \ar[r] & \myR^{\geq p+1}\phi_*\sO_{\wt X}(-\wt\DD)[-1] \ar[r]^-{+1} & .  }
$$

\subsection{Easy vanishing theorems for $\sExt$ sheaves}

\begin{lem}\label{lem:ext-vanish=p}
  Let $(X,\DD)$ be a pair and $\phi:(\wt X,\wt \DD)\to (X,\DD)$ a log resolution. Let
  $p>0$ be a positive integer and $x\in X$ which is not an associated point of
  $\myR^p\phi_*\sO_{\wt X}(-\wt\DD)$. Then
  $$
  \left(\sExt_X^{j}(\myR^p\phi_*\sO_{\wt X}(-\wt\DD), \omega_X^\mydot)\right)_x=0,
  \quad\text{ for and $j\geq -\dim x$.}
  $$
\end{lem}

\begin{proof}
  As $x\in X$ is not an associated point of $\myR^p\phi_*\sO_{\wt X}(-\wt\DD)$, it follows
  that either $\left(\myR^p\phi_*\sO_{\wt X}(-\wt\DD)\right)_x=0$ or
  $\depth_x\myR^p\phi_*\sO_{\wt X}(-\wt\DD)\geq 1$. In the latter case the statement
  follows from \eqref{thm:depth-ext-criterion}.
\end{proof}

\begin{cor}\label{cor:ext-vanish-at-least-p}
  Let $(X,\DD)$ be a pair and $\phi:(\wt X,\wt \DD)\to (X,\DD)$ a log resolution. Let
  $p>0$ be a positive integer and $x\in X$ which is not an associated point of
  $\myR^i\phi_*\sO_{\wt X}(-\wt\DD)$ for any $i\geq p$. Then
  $$
  \left(\sExt_X^{j}(\myR^{\geq p}\phi_*\sO_{\wt X}(-\wt\DD),
    \omega_X^\mydot)\right)_x=0, \quad\text{ for $j\geq -\dim x$.}
  $$  
\end{cor}

\begin{proof}
  If $p\gg 0$ then $\myR^{p}\phi_*=\myR^{\geq p}\phi_*=0$, and hence the statement
  holds trivially. Next apply the functor $\myR\sHom_X(\blank, \omega_X^\mydot)$ to
  the distinguished triangle:
  $$
  \xymatrix{%
    \myR^p\phi_*\sO_{\wt X}(-\wt\DD) \ar[r] & \myR^{\geq p}\phi_*\sO_{\wt X}(-\wt\DD)
    \ar[r] & \myR^{\geq p+1}\phi_*\sO_{\wt X}(-\wt\DD)[-1] \ar[r]^-{+1} & }
  $$
  to obtain the long exact sequence,
  \begin{multline*}
    \dots\to \sExt^{j}_X(\myR^{\geq p+1}\phi_*\sO_{\wt X}(-\wt\DD)[-1],
    \omega_X^\mydot) \simeq \sExt^{j+1}_X(\myR^{\geq p+1}\phi_*\sO_{\wt X}(-\wt\DD),
    \omega_X^\mydot) \to \\
    \to \sExt^{j}_X(\myR^{\geq p}\phi_*\sO_{\wt X}(-\wt\DD), \omega_X^\mydot) \to
    \sExt^{j}_X(\myR^p\phi_*\sO_{\wt X}(-\wt\DD) , \omega_X^\mydot) \to \dots.
  \end{multline*}
  Now the statement follows from \eqref{lem:ext-vanish=p} by descending induction on
  $p$.
\end{proof}

\subsection{Main theorem}

\begin{thm}\label{thm:main-theorem}
  Let $(X,\DD)$ be a reduced pair, $\phi:(\wt X,\wt \DD)\to (X,\DD)$ a log
  resolution, and $x\in X$ which is not the general point of a relative irrational
  center of $(X,\DD)$ with respect to $\phi$. Then
  \begin{equation}
    \label{eq:2}
    \left(\sExt_X^{-i}(\phi_*\sO_{\wt X}(-\wt\DD), \omega_X^\mydot)\right)_x=0,
    \quad\text{ for $i<\min ( 3, \codim_X x )+\dim x$,}
  \end{equation}
  in particular
  \begin{equation}
    \label{eq:3}
    \depth_x\phi_*\sO_{\wt X}(-\wt\DD)\geq \min ( 3, \codim_X x ).
  \end{equation}
\end{thm}

\begin{proof}
  First observe that $\dim_x\phi_*\sO_{\wt X}(-\wt\DD)=\codim_Xx$, so (\ref{eq:2})
  implies (\ref{eq:3}) by \eqref{thm:depth-ext-criterion}. I will prove (\ref{eq:2}).
  Consider the distinguished triangle:
  $$
  \xymatrix{%
    \phi_*\sO_{\wt X}(-\wt\DD) \ar[r] & \myR\phi_*\sO_{\wt X}(-\wt\DD) \ar[r] &
    \myR^{\geq 1}\phi_*\sO_{\wt X}(-\wt\DD)[-1] \ar[r]^-{+1} & ,  }
  $$
  and the long exact sequence it induces:
  \begin{multline*}
    \dots \to \sExt^{-i}_X(\myR\phi_*\sO_{\wt X}(-\wt\DD), \omega_X^\mydot) \to
    \sExt^{-i}_X(\phi_*\sO_{\wt X}(-\wt\DD) , \omega_X^\mydot) \to \\ \to
    \sExt^{-i+2}_X(\myR^{\geq 1}\phi_*\sO_{\wt X}(-\wt\DD), \omega_X^\mydot) \to
    \dots.
  \end{multline*}  
  If $i<3+\dim x$, then $-i+2\geq -\dim x$ so $\left(\sExt^{-i+2}_X(\myR^{\geq
      1}\phi_*\sO_{\wt X}(-\wt\DD), \omega_X^\mydot)\right)_x=0$ by
  \eqref{cor:ext-vanish-at-least-p} and so it is enough to prove that
  $\left(\sExt^{-i}_X(\myR\phi_*\sO_{\wt X}(-\wt\DD), \omega_X^\mydot)\right)_x=0$
  for $i<\min(3,\codim_Xx)$.

  Let $n=\dim X$ and observe that 
  \begin{multline}
    \sExt^{-i}_X(\myR\phi_*\sO_{\wt X}(-\wt\DD), \omega_X^\mydot) \simeq
    h^{-i}(\myR\sHom_X(\myR\phi_*\sO_{\wt X}(-\wt\DD), \omega_X^\mydot)) \simeq \\
    \simeq h^{-i}(\myR\phi_*\myR\sHom_{\wt X}(\sO_{\wt X}(-\wt\DD), \omega_{\wt
      X}^\mydot)) \simeq h^{-i}(\myR\phi_* \omega_{\wt X}^\mydot(\wt\DD)) \simeq
    \myR^{n-i}\phi_* \omega_{\wt X}(\wt\DD)
  \end{multline}
  by Grothendieck duality and the fact that $\wt X$ is smooth and $\wt \DD$ is a
  Cartier divisor.  Now observe that if $i<n$, then $\left(\myR^{n-i}\phi_*
    \omega_{\wt X}(\wt\DD)\right)_x=0$ by \eqref{thm:rel-kv-vanish}.
\end{proof}

\begin{cor}\label{thm:appl--theorem}
  Let $(X,\DD)$ be a normal pair and $x\in X$ which is not the general point of an
  absolute irrational center of $(X,\DD)$. Then
  $$\depth_x\sO_X(-\DD)\geq \min ( 3, \codim_X x ). $$
\end{cor}

\begin{proof}
  Let $\phi:(\wt X,\wt \DD)\to (X,\DD)$ be a log resolution for which $x\in X$ is not
  the general point of a relative irrational center of $(X,\DD)$ with respect to
  $\phi$. Notice that as $(X,\DD)$ is normal, by definition,
  $$
  \sO_X(-\DD)\simeq \phi_*\sO_{\wt X}(-\wt\DD),
  $$
  so the statement is straightforward from \eqref{thm:main-theorem}.
\end{proof}

\section{Applications to log canonical pairs}

The key point of applying the results of this paper to log canonical pairs is that
absolute irrational centers are non-klt centers.  It is easy to see that the union of
all non-klt centers of a log canonical pair contains the locus where that log
canonical pair is not rational and hence it contains the union of all absolute
irrational centers.  However, I am claiming that there is a closer relationship,
namely that the absolute irrational centers themselves are non-klt centers.

Next we will discuss the key step in applying the theory of irrational centers to log
canonical pairs.  We will use the following slight abuse of notation: For a log
resolution $\phi:(Y,{\Delta_Y})\to (X,\Delta)$ we will denote the log resolution
$\phi:(Y,\rdown{\Delta_Y})\to (X,\rdown\Delta)$ by the same symbol. This makes sense
as $\phi$ really stands for the birational morphism $\phi:Y\to X$ that, as a
morphism, is a priori independent of the choice of boundary divisor.

We will need the following in the proof.

\begin{defn}\label{def:szabo-res}
  Let $(Z,\Theta)$ be a dlt pair.  A log resolution of $(Z,\Theta)$, $g:(Y,
  \Gamma)\to (Z,\Theta)$ is called a \emph{\szabores}, if there exist $A,B$ effective
  $\bQ$-divisors on $Y$ without common irreducible components, such that $\supp
  (A+B)\subset \exc (g)$, $\rdown{A}=0$,
  and
  $$
  K_{Y}+\Gamma \sim_{\bQ} g^*(K_Z+\Theta) - A + B.
  $$
\end{defn}

\begin{rem}
  Every dlt pair admits a \szabores by \cite{MR1322695} (cf.\ \cite[2.44]{KM98}).
\end{rem}

\begin{thm}\label{thm:irrational-is-non-klt}
  Let $(X,\Delta)$ be a log canonical pair.  Then there exists a log resolution
  $\phi:(Y,\Delta_Y)\to (X,\Delta)$ such that every relative irrational center of
  $(X,\rdown\Delta)$ with respect to $\phi$ is the image of a non-klt center of
  $(Y,\Delta_Y)$. In particular, every absolute irrational center of
  $(X,\rdown\Delta)$ is also a non-klt center of $(X,\Delta)$.
\end{thm}

\begin{proof}
  First let $\psi:(Z,\Delta_Z+E) \to (X,\Delta)$ be a minimal dlt model
  \cite[3.1]{KK10}, where $\Delta_Z=\psi_*^{-1}\Delta$ and $E=\exc(\psi)$. By the
  definition of the minimal dlt model $Z$ is $\bQ$-factorial and
  \begin{equation}
    \label{eq:6}
    K_Z + \Delta_Z+E  \sim_{\bQ} \psi^*(K_X + \Delta).
  \end{equation}
  Next let $\eta: (Y,\Delta_Y+E_Y)\to (Z,\Delta_Z+E)$ be a \szabores
  \eqref{def:szabo-res}, where $\Delta_Y= \eta_*^{-1}\Delta_Z$ and $E_Y=
  \eta_*^{-1}E$. Then there are $A,B$ effective $\bQ$-divisors on $Y$ without common
  irreducible components, such that $\supp (A+B)\subset \exc (\eta)$, $\rdown{A}=0$,
  and
  \begin{equation}
    \label{eq:5}
    K_{Y}+\Delta_Y+E_Y \sim_{\bQ} \eta^*(K_Z+\Delta_Z+E) - A + B.
  \end{equation}
  Let $B_\varepsilon\leteq B + \varepsilon(\eta^*E-E_Y)$ where $0<\varepsilon \ll 1$.
  It follows that
  \begin{multline}
    \label{eq:1}
    \rup {B_\varepsilon} - \rdown{\Delta_Y} \sim_{\bQ} K_Y+ \{\Delta_Y\}+
    (1-\varepsilon)E_Y + A + \{-B_\varepsilon\} + \\ 
    -\eta^*(K_Z+\Delta_Z+(1-\varepsilon)E),
  \end{multline}
  and then the relative Kawamata-Viehweg vanishing theorem implies that
  \begin{equation*}
    \myR^i\eta_*\sO(\rup{B_\varepsilon} -\rdown{\Delta_Y}) =0 \quad \text{ for $i>0$.}
  \end{equation*}
  As $B_\varepsilon$ is still an effective $\eta$-exceptional divisor, it follows
  that $$\eta_*\sO_Y(\rup {B_\varepsilon} -\rdown{\Delta_Y})\simeq
  \sO_Z(-\rdown{\Delta_Z}),$$ cf.\ \eqref{prop:normal-reduced-is-a-normal-pair}, and
  hence
  \begin{equation}
    \label{eq:4}
    \myR\eta_*\sO_Y(\rup {B_\varepsilon} -\rdown{\Delta_Y}) \qis
    \sO_Z(-\rdown{\Delta_Z}). 
  \end{equation}

  \noindent
  Clearly, $\phi=\psi\circ\eta$ is a log resolution of $(X,\Delta)$ and $\Delta_Y=
  \phi_*^{-1}\Delta$.  Then from (\ref{eq:6}) and (\ref{eq:1}) we obtain
  \begin{equation}
    \label{eq:9}
    \rup {B_\varepsilon} 
    - \rdown{\Delta_Y} \sim_{\bQ} K_Y+ \{\Delta_Y\}+ E_Y + A + \{-B_\varepsilon\} 
    + \varepsilon(\eta^*E-E_Y) 
    -\phi^*(K_X+\Delta).
  \end{equation}
  Note that as opposed to (\ref{eq:1}), here the boundary relative to $\phi$ may have
  components with coefficient $1$, namely $E_Y$, so we cannot apply the relative
  Kawamata-Viehweg vanishing theorem with respect to $\phi$.

  On the other hand, notice that since $(Z,\Delta_Z+E)$ is $\mathbb Q$-factorial and
  dlt, so is $(Z,\Delta_Z)$ and then by \cite[2.1]{Kovacs10a}
  $\myR^i\eta_*\sO_Y(-\rdown{\Delta_Y})=0$ for $i>0$. Therefore, using (\ref{eq:4}),
  we obtain that
  \begin{multline*}
    \myR\phi_*\sO_Y(-\rdown{\Delta_Y}) \qis
    \myR\psi_*\myR\eta_*\sO_Y(-\rdown{\Delta_Y}) \qis
    \myR\psi_*\sO_Z(-\rdown{\Delta_Z}) \qis \\ \qis \myR\psi_*\myR\eta_*\sO_Y(\rup
    {B_\varepsilon} -\rdown{\Delta_Y}) \qis \myR\phi_*\sO_Y(\rup {B_\varepsilon}
    -\rdown{\Delta_Y}).
  \end{multline*}
  Next use Koll\'ar-Ambro torsion freeness: Applying \cite[2.39]{Fujinobook} to
  (\ref{eq:9}) we obtain that the closure of any associated point of
  $$\myR^i\phi_*\sO_Y(-\rdown{\Delta_Y}) \simeq\myR^i\phi_*\sO_Y(\rup{B_\varepsilon}
  -\rdown{\Delta_Y})$$ for some $i>0$ is a non-klt center of $(X, \Delta)$.
  Similarly, the closure of an associated point of
  $\myR^i\phi_*\omega_{Y}(\rdown{\Delta_Y})$ for some $i>0$ is a non-klt center of
  $(X, \Delta)$ by \cite[2.39]{Fujinobook}.
\end{proof}

As indicated in the introduction we obtain the following corollary:

\begin{cor}\label{cor:lc}
  Let $(X,\Delta)$ be a log canonical pair and $x\in X$ which is not the general
  point of a non-klt center of $(X,\Delta)$. Then
  $$\depth_x\sO_X(-\rdown{\Delta})\geq \min ( 3, \codim_X x ). $$
\end{cor}

\begin{proof}
  By \eqref{thm:irrational-is-non-klt} there exists a log resolution
  $\phi:(Y,\Delta_Y)\to (X,\Delta)$ such that $x$ is not a relative irrational center
  of $(X,\rdown\Delta)$ with respect to $\phi$. Then
  $$\depth_x\phi_*\sO_Y(-\rdown{\Delta_Y})\geq \min ( 3, \codim_X x )$$ by
  \eqref{thm:main-theorem} and $\sO_{X}(-\rdown{\Delta})\simeq
  \phi_*\sO_{Y}(-\rdown{\Delta_Y})$ by \eqref{prop:normal-reduced-is-a-normal-pair}.
\end{proof}

\begin{ack}
  I would like to thank Valery Alexeev, Zsolt Patakfalvi, and Karl Schwede for
  helpful comments.
\end{ack}


\def\cprime{$'$} \def\polhk#1{\setbox0=\hbox{#1}{\ooalign{\hidewidth
  \lower1.5ex\hbox{`}\hidewidth\crcr\unhbox0}}} \def\cprime{$'$}
  \def\cprime{$'$} \def\cprime{$'$} \def\cprime{$'$}
  \def\polhk#1{\setbox0=\hbox{#1}{\ooalign{\hidewidth
  \lower1.5ex\hbox{`}\hidewidth\crcr\unhbox0}}} \def\cdprime{$''$}
  \def\cprime{$'$} \def\cprime{$'$} \def\cprime{$'$} \def\cprime{$'$}
\providecommand{\bysame}{\leavevmode\hbox to3em{\hrulefill}\thinspace}
\providecommand{\MR}{\relax\ifhmode\unskip\space\fi MR}
\providecommand{\MRhref}[2]{%
  \href{http://www.ams.org/mathscinet-getitem?mr=#1}{#2}
}
\providecommand{\href}[2]{#2}

\end{document}